\documentclass[10pt,english,reqno]{amsart}
\usepackage[T1]{fontenc}
\usepackage[latin9]{inputenc}
\usepackage{amsthm}
\usepackage{amssymb}

\makeatletter
\numberwithin{equation}{section}
\numberwithin{figure}{section}
  \theoremstyle{plain}
  \newtheorem*{conjecture*}{\protect\conjecturename}
\theoremstyle{plain}
\newtheorem{thm}{\protect\theoremname}
  \theoremstyle{plain}
  \newtheorem{prop}[thm]{\protect\propositionname}
  \theoremstyle{plain}
  \newtheorem{cor}[thm]{\protect\corollaryname}
  \theoremstyle{remark}
  \newtheorem{rem}[thm]{\protect\remarkname}
  \theoremstyle{plain}
  \newtheorem{lem}[thm]{\protect\lemmaname}

\usepackage[matrix,arrow,tips,curve,ps]{xy}\SelectTips{cm}{10}
\usepackage{enumerate}
\usepackage{comment} 

\theoremstyle{definition}

\numberwithin{Theorem}{section} \numberwithin{equation}{section}

\def\CC{{\mathbb C}}
\def\DD{{\mathbb D}}

\def\FF{{\mathbb F}}

\def\HH{{\mathbb H}}

\def\PP{{\mathbb P}}
\def\QQ{{\mathbb Q}}
\def\RR{{\mathbb R}}

\def\ZZ{{\mathbb Z}}

\newcommand{\NS}{\operatorname{NS}}

\makeatother

\usepackage{babel}
  \providecommand{\conjecturename}{Conjecture}
  \providecommand{\corollaryname}{Corollary}
  \providecommand{\lemmaname}{Lemma}
  \providecommand{\propositionname}{Proposition}
  \providecommand{\remarkname}{Remark}
\providecommand{\theoremname}{Theorem}

\begin{document}

\newcommand{\Gal}{{\rm Gal}} 
\newcommand{\Pic}{{\rm Pic}} 

\title[Bounded negativity, Miyaoka-Sakai inequality]{Bounded negativity, Miyaoka-Sakai inequality and elliptic curve configurations}

\addtolength{\textwidth}{0mm}
\addtolength{\hoffset}{-0mm} 
\addtolength{\textheight}{0mm}
\addtolength{\voffset}{-0mm} 

\global\long\global\long\def\Alb{{\rm Alb}}
 \global\long\global\long\def\Jac{{\rm Jac}}
 \global\long\global\long\def\Hom{{\rm Hom}}
 \global\long\global\long\def\End{{\rm End}}
 \global\long\global\long\def\aut{{\rm Aut}}
 \global\long\global\long\def\NS{{\rm NS}}
 \global\long\global\long\def\SSm{{\rm S}}
 \global\long\global\long\def\psl{{\rm PSL}}
 \global\long\global\long\def\CC{\mathbb{C}}
 \global\long\global\long\def\BB{\mathbb{B}}
 \global\long\global\long\def\PP{\mathbb{P}}
 \global\long\global\long\def\QQ{\mathbb{Q}}
 \global\long\global\long\def\RR{\mathbb{R}}
 \global\long\global\long\def\FF{\mathbb{F}}
 \global\long\global\long\def\DD{\mathbb{D}}
 \global\long\global\long\def\NN{\mathbb{N}}
 \global\long\global\long\def\ZZ{\mathbb{Z}}
 \global\long\global\long\def\HH{\mathbb{H}}
 \global\long\global\long\def\gal{{\rm Gal}}
 \global\long\global\long\def\OO{\mathcal{O}}
 \global\long\global\long\def\pP{\mathfrak{p}}
 \global\long\global\long\def\pPP{\mathfrak{P}}
 \global\long\global\long\def\qQ{\mathfrak{q}}
 \global\long\global\long\def\mm{\mathcal{M}}
 \global\long\global\long\def\aaa{\mathfrak{a}}

\author{Xavier Roulleau}
\begin{abstract}
Similarly to the linear Harbourne constant recently introduced in
\cite{BDHH}, we study the elliptic $H$-constants of $\mathbb{P}^{2}$
and Abelian surfaces. We exhibit configurations of smooth plane cubic
curves whose Harbourne index is arbitrarily close to $-4$. As a Corollary,
we obtain that the global $H$-constant of any surface $X$ is less
or equal to $-4$. Related to these problems, we moreover give a new
inequality for the number and multiplicities of singularities of elliptic
curves arrangements on Abelian surfaces, inequality which has a close
similarity to the one of Hirzebruch for arrangements of lines in the
plane.
\end{abstract}
\maketitle

\section*{Introduction}

The Harbourne constant (for short $H$-constant) of a surface and
some related variants of it have been recently introduced in \cite{BDHH},
bringing new problems and open questions on curves and their singularities
on surfaces. As explained there, these constants measure the local
negativity of curves on surfaces, in analogy with the local positivity
measured by Seshadri constants. These constants come from the context
of the Bounded Negativity Conjecture (BNC):
\begin{conjecture*}
\label{Conjecture BNC}Let $X_{/\CC}$ be a smooth projective surface.
There exists an integer $b(X)$ such that for every (reduced) curve
$C$ on $X$, we have $C^{2}\geq-b(X)$.
\end{conjecture*}
This ancient and now intensively studied conjecture (\cite{BDHH}, \cite{BHK}, \cite{BPS}, \cite{Ciliberto}, \cite{Koziarz}, \cite{Moeller})
is trivially true for the plane, but we do not know the behavior of
the problem if one takes blow-ups of it. The $H$-constant was introduced
to approach that question. For a blow-up $X'\to X$ of a surface $X$
at a set $\mathcal{P}$ of $s>0$ distinct points and $C\hookrightarrow X$
a curve, we denote by $\bar{C}$ the strict transform of $C$ in $X'$.
The $H$-constant of $X$ is defined by 
\[
H_{X}:=\inf_{C,\mathcal{P}}H(C,\mathcal{P})
\]
 where 
\[
H(C,\mathcal{P})=\frac{(\bar{C})^{2}}{s}
\]
for $C\hookrightarrow X$ varying over all reduced curves on $X$,
and the strict transform of $C$ varying over the blows-ups at sets
$\mathcal{P}$ of $s$ distinct points on $X$ ($s>0$ any). If finite,
the $H$-constant has the interesting property that whenever BNC holds
for $X$, then it holds for any blow-ups of it. 

Variations of the $H$-constants are proposed in \cite{BDHH}, among
them the linear $H$-constant $H_{L}$ for the plane. Here $X=\PP^{2}$
and the infimum $H_{L,\PP^{2}}=\inf_{C,\mathcal{P}}H(C,\mathcal{P})$
is taken over every union $C$ of lines in $\PP^{2}$. Using Hirzebruch
bounds on the singularities of lines configurations, the authors proves
that $H_{L,\PP^{2}}\geq-4.$ They moreover found an example of union
of lines $C$ such that its Harbourne index
\[
H(C)=\min_{\mathcal{P}}H(C,\mathcal{P})
\]
is very negative: $H(C)=-225/67$, therefore one knows that $-225/67\geq H_{L,\PP^{2}}\geq-4$,
and $-225/67\geq H_{L,\PP^{2}}\geq H_{\PP^{2}}$. 

In this paper, we similarly propose to study the elliptic $H$-constant
$H_{El,X}$ defined by 
\[
H_{El,X}=\inf_{C,\mathcal{P}}H(C,\mathcal{P})=\inf_{C}H(C)
\]
for unions $C$ of smooth elliptic curves on the surface $X$. Let
$C$ be such a configuration of elliptic curves on an Abelian surface
$A$. Like for an arrangement of lines in the plane, the curve $C$
has only ordinary singularities. For $k\geq2$, we denote by $t_{k}$
be the number of $k$-points on $C$ i.e. the points with multiplicity
$k$. Let be $f_{0}=\sum t_{k}$ and $f_{1}=\sum kt_{k}$. 
\begin{thm}
\label{prop:elliptic h const >-4}For any set $\mathcal{P}$ of points,
we have 
\begin{equation}
H(C,\mathcal{P})\geq H(C)=-\frac{f_{1}}{f_{0}}\geq\frac{t_{2}+\frac{1}{4}t_{3}}{f_{0}}-4\geq-4\label{eq:thm3}
\end{equation}
The elliptic H-constant satisfies $H_{El,A}\geq-4$. The $H$-index
of $C$ equals $-4$ if and only if the only singularities on $C$
are $4$-points.\\
The $H$-constant and elliptic $H$-constant are isogeny invariants
: if $A$ and $B$ are isogenous, then $H_{A}=H_{B}$ and $H_{El,A}=H_{El,B}$. 
\end{thm}
Let $A$ be either the surface $(\CC/\ZZ[\alpha])^{2}$, where $\alpha^{2}+\alpha+1=0$
or $(\CC/\ZZ[i])^{2}$, $i^{2}=-1$. By the constructions of Hirzebruch
\cite{Hirz84} and Holzapfel \cite[Example 5.4]{Holzapfel}, there
exist configurations $C$ of elliptic curves on $A$ such that their
$H$-index satisfy $H(C)=-4$. We thus see that the bound $H_{El,A}\geq-4$
is optimal. Actually, by Kobayashi's results \cite{Kobayashi}, the
equality $H(C)=-4$ is attained if and only if the complement of $C$
in $A$ is birational to an open ball quotient surface. Inequality
\[
H(C)=-\frac{f_{1}}{f_{0}}\geq\frac{t_{2}+\frac{1}{4}t_{3}}{f_{0}}-4
\]
in Theorem \ref{prop:elliptic h const >-4} is a consequence of a
more general result (Theorem \ref{thm:Refined Miyaoka Sakai inequality})
on the $H$-constant of curve configurations of any genus proved by
using $(\ZZ/n\ZZ)^{d}$-covers of Abelian surfaces. We obtain in particular
the following result:
\begin{thm}
\label{thm:For-a-configurationMain}For a configuration of elliptic
curves on an Abelian surface, one has 
\[
t_{2}+\frac{3}{4}t_{3}\geq\sum_{k\geq5}(2k-9)t_{k}.
\]

\end{thm}
This is the very analog of the inequality $t_{2}+\frac{3}{4}t_{3}\geq d+\sum_{k\geq5}(2k-9)t_{k}$
due to Hirzebruch for a configuration of $d\geq6$ lines in $\mathbb{P}^{2}$
such that $t_{d}=t_{d-1}=t_{d-2}=0$ (see \cite[eq. (9)]{Hirz86}).

We also obtain the following results for the elliptic and global $H$-constants
of the plane:
\begin{prop}
\label{thm:configurations}There exist configurations $C_{n}$ of
smooth cubic curves in $\PP^{2}$ such that for $Sing(C_{n})$ the
set of singular points of $C_{n}$, we have 
\[
\lim_{n}H(C_{n},Sing(C_{n}))=-4,
\]
and therefore $-4\geq H_{El,\PP^{2}}\geq H_{\PP^{2}}$.
\end{prop}
The configurations $C_{n}$ appear in \cite{RU}. From Proposition
\ref{thm:configurations} and functorial properties of the $H$-constants,
we get the following Corollary on the Harbourne constant of any surface:
\begin{cor}
\label{cor:HX leq -4}Let $X$ be a smooth surface. Then $-4\geq H_{\PP^{2}}\geq H_{X}$.
\end{cor}
The paper is organized as follows. In the first section, we prove
Proposition \ref{thm:configurations} and Corollary \ref{cor:HX leq -4}
concerning the elliptic $H$-constant and the Harbourne constant of
the plane. In the second section we prove Theorem \ref{thm:For-a-configurationMain},
which together with Theorem \ref{prop:elliptic h const >-4} proved
in the third section are the main results. In the last section we
discuss some questions and problems raised by the definitions of the
$H$-index of a curve and $H$-constant of a surface.

\textbf{Acknowledgements}. The author thanks Piotr Pokora for its
relevant remarks on a previous version of this paper.

\section{elliptic curve configurations on the plane}

Let $p:Z\to\PP^{2}$ be the blow-up of $\PP^{2}$ at the $12$ singular
points $\mathcal{P}_{HE}=\{p_{1},\dots,p_{12}\}$ of the dual Hesse
configuration. In \cite{RU}, is described some elliptic curve configurations
$\mathcal{H}_{n}$ ($n\in3\NN,\, n>0$) on $Z$ with the following
properties:\\
i) $\mathcal{H}_{n}$ is the union of $\frac{4}{3}(n^{2}-3)$ elliptic
curves which are fibers of some elliptic fibrations of $Z$,\\
ii) Each elliptic curve of $\mathcal{H}_{n}$ has $9$ $3$-points
and $n^{2}-9$ $4$-points,\\
iii) The singularities of $\mathcal{H}_{n}$ are $\frac{1}{3}(n^{2}-3)(n^{2}-9)$
$4$-points and $4(n^{2}-3)$ $3$-points.

Let $C_{n}$ be the image in $\PP^{2}$ of $\mathcal{H}_{n}$ by the
blow-down map $p$. Then :\\
i) $C_{n}$ is the union of $\frac{4}{3}(n^{2}-3)$ smooth degree
$3$ curves.\\
ii) Each elliptic curve contains $n^{2}-9$ $4$-points and goes through
$9$ of the $12$ singular points of the dual Hesse configuration.\\
iii) The singularities of $C_{n}$ are $\frac{1}{3}(n^{2}-3)(n^{2}-9)$
$4$-points and the $12$ points in $\mathcal{P}_{HE}$ have multiplicity
$(n^{2}-3)$ (on the blow up at $p_{i}\in\mathcal{P}_{HE}$, the strict
transform of the curve $C_{n}$ has $\frac{1}{3}(n^{2}-3)$ $3$-points
on the exceptional divisor).
\begin{prop}
\label{prop:There-exists-a}The configurations $C_{n}$ of smooth
cubic curves on $\PP^{2}$ satisfy 
\[
\lim_{n}H(C_{n},\mathcal{P}_{n})=-4,
\]
 where $\mathcal{P}_{n}$ is the set of singular points of $C_{n}$.\end{prop}
\begin{proof}
Let be $X_{n}\to\PP^{2}$ be the blow-up of $\PP^{2}$ at $\mathcal{P}_{n}$,
the $s_{n}=12+\frac{1}{3}(n^{2}-3)(n^{2}-9)$ singularities of $C_{n}$.
Let $\bar{C}_{n}$ be the strict transform of $C_{n}$. The blow down
map $X_{n}\to\PP^{2}$ factorizes through $Z$. Then 
\[
\bar{C}_{n}^{2}=\sum\bar{C}_{i}^{2}+2\sum_{i<j}\bar{C}_{i}\bar{C}_{j}
\]
where $\bar{C_{i}}$ is the strict transform of a curve $C_{i}$ in
$\mathcal{H}_{n}$. Since $\bar{C}_{i}^{2}=-(n^{2}-9)$ and since
from the configuration one has: 
\[
\sum_{i<j}\bar{C}_{i}\bar{C}_{j}=4(n^{2}-3),
\]
we get 
\[
\bar{C}_{n}^{2}=-(n^{2}-9)\times\frac{4}{3}(n^{2}-3)+8(n^{2}-3)=-\frac{4}{3}n^{4}+O(n^{3}).
\]
As $s_{n}=\frac{1}{3}n^{4}+O(n^{3})$, we obtain the result. \end{proof}
\begin{rem}
\label{Remarque The-above-configurations}A) By example, for $n=7,$
we obtain $H(C_{n},\mathcal{P}_{n})=-\frac{20148}{5257}\simeq-3,83.$
\\
B) The above configurations $C_{n}$ are strongly linked to some compactifications
$X_{n}$ of some open ball quotient surfaces (\cite{Hirz84}), for
which $\lim\frac{c_{1}^{2}}{c_{2}}(X_{n})=3$, ie one is close to
the upper bound in the Miyaoka-Yau inequality. 
\end{rem}
Let $f:X\to Y$ be a dominant morphism between two smooth surfaces.
Let $C$ be a reduced curve on $Y$. Suppose that $C$ do not contain
components in the branch locus $B$ and let $\mathcal{P}$ be a set
of $s$ point in $Y$ disjoint with $B$ (so that $f^{*}C$ and $f^{*}\mathcal{P}$
are reduced of pure dimension $1$ and $0$ respectively). 
\begin{lem}
\label{lem:Then branch}Then $H_{X}(f^{*}C,f^{*}\mathcal{P})=H_{Y}(C,\mathcal{P}).$\end{lem}
\begin{proof}
Let $d$ be the degree of $f$. Let $p$ be a point of $\mathcal{P}$.
Then the $d$ points above $p$ have the same multiplicity $m_{p}$
inside $C'=f^{*}C$ than $p$ inside $C$ and we have:
\[
H_{X}(f^{*}C,f^{*}\mathcal{P})=\frac{\bar{C'}^{2}}{ds}=\frac{dC{}^{2}-d\sum m_{p}^{2}}{ds}=\frac{\bar{C}^{2}}{s}=H_{Y}(C,\mathcal{P}).
\]

\end{proof}
Proposition \ref{prop:There-exists-a} implies :
\begin{cor}
Let $X$ be a smooth surface. Then $H_{X}\leq H_{\PP^{2}}\leq-4$.\end{cor}
\begin{proof}
Let $f:X\to\PP^{2}$ be a generic projection of $X$ onto the plane.
For any curve $C$ in $\PP^{2}$ and set of point $\mathcal{P}\subset\PP^{2}$,
we can choose an automorphism $g$ such that $g^{*}C$ do not contain
any components of the branch divisor $B$ and $g^{*}\mathcal{P}$
is disjoint with $B$. We then apply Lemma \ref{lem:Then branch}
and results of Proposition \ref{prop:There-exists-a} to conclude
that $H_{X}\leq H_{\PP^{2}}\leq-4$.
\end{proof}

\section{Arrangement of curves on Abelian surfaces and $(\ZZ/n\ZZ)^{d}$-covers.}

Let $A$ be an abelian surface and let $C=\sum_{i=1}^{d}C_{i}$ be
a reduced divisor with only ordinary singularities, a union of $d\geq2$
smooth divisors $C_{i}$ (by example $C_{i}$ may be the union of
genus $1$ fibers of a fibration of $A$ onto an elliptic curve).
Let $g$ be the geometric genus of $C$ ie 
\[
g-1=\sum g_{j}-1,
\]
where the $g_{j}$'s are the geometric genus of the irreducible components
of $C$. Let us denote by $t_{k}$ the number of $k$-points on $C$.
\begin{thm}
\label{thm:Refined Miyaoka Sakai inequality}We have
\[
10g-10+t_{2}+\frac{3}{4}t_{3}\geq\sum_{k\geq5}(2k-9)t_{k},
\]
and
\begin{equation}
H(C,Sing(C))=\frac{2g-2-f_{1}}{f_{0}}\geq\frac{2t_{2}+\frac{9}{8}t_{3}+\frac{1}{2}t_{4}+8-8g}{f_{0}}-\frac{9}{2}.\label{eq:for ell curves}
\end{equation}
Suppose that $C$ is a configuration of elliptic curves. We have 
\begin{equation}
H(C)=-\frac{f_{1}}{f_{0}}\geq\frac{t_{2}+\frac{1}{4}t_{3}}{f_{0}}-4\geq-4\label{eq:H(C) plus grand t2+1/2t3}
\end{equation}
where $H(C)$ is the $H$-index: $H(C)=\min_{\mathcal{P}}H(C,\mathcal{P})$.
Moreover $H(C)=-4$ if and only if the only singularities on $C$
are $4$-points.
\end{thm}
The remaining of this section is the proof of Theorem \ref{thm:Refined Miyaoka Sakai inequality}.
Let us recall a Theorem of Namba on branched covers. Let $M$ be a
manifold, $D_{1},\dots,D_{s}$ be irreducible reduced divisors on
$M$ and $n_{1},\dots,n_{s}$ be positive integers. We denote by $D$
the divisor $D=\sum n_{i}D_{i}$. Let $Div(M,D)$ be the sub-group
of the $\QQ$-divisors generated by the entire divisors and 
\[
\frac{1}{n_{1}}D_{1},\dots,\frac{1}{n_{s}}D_{s}.
\]
Let $\sim$ be the linear equivalence in $Div(M,D)$, where $G\sim G'$
if and only if $G-G'$ is an entire principal divisor. Let $Div(M,D)/\sim$
the quotient and let $Div^{0}(M,D)/\sim$ be the kernel of the Chern
class map 
\[
\begin{array}{ccc}
Div(M,D)/\sim & \to & H^{1,1}(M,\RR)\\
G & \to & c_{1}(G)
\end{array}.
\]

\begin{thm}
\label{thm:(Namba).-There-exists}(Namba, \cite[Thm. 2.3.20]{Namba}).
There exists a finite Abelian cover which branches at D with index
$n_{i}$ over $D_{i}$ for all $i=1,\dots,s$ if and only if for every
$j=1,\dots,s$ there exists an element of finite order $v_{j}=\sum\frac{a_{ij}}{n_{i}}D_{i}+E_{j}$
of $Div^{0}(M,D)/\sim$ ($E_{j}$ an entire divisor, $a_{ij}\in\ZZ$)
such that $a_{jj}$ is coprime to $n_{j}$. \\
Then the group in $Div^{0}(M,D)/\sim$ generated by the $v_{j}$ is
isomorphic to the Galois group of one of such Abelian cover.
\end{thm}
We find the inequalities among the $t_{k}$'s in Theorem \ref{thm:Refined Miyaoka Sakai inequality}
using $(\ZZ/n\ZZ)^{d}$-covers of $A$ ramified above curves (related
to the) $C_{i}$. These inequalities involve quantities that are ``linear''
under isogenies, by which we mean that if $\phi:B\to A$ is an isogeny
of degree $m$, then the number of $k$-points on $\phi^{*}C$ (a
reduced curve), the intersections between the $\phi^{*}C_{i}\mbox{'s and }\phi^{*}C$,
the geometric genus minus $1$ of $\phi^{*}C$ etc... are the ones
of $C,\, C_{i}$ etc... multiplied by $m$. By that property, inequalities
involving linear terms in the $t_{k}$'s and $C_{i}^{2}$ proved on
abelian surface $B$ are then inequalities for $A$.\\
 Recall that if $\phi=[m]:A\to A$ is the multiplication by $m\in\NN$
map, then $\phi^{*}D\sim\frac{m(m+1)}{2}D+\frac{m(m-1)}{2}[-1]^{*}D$
for any divisor $D$. By taking $m=2n,$ one can therefore suppose
that the divisors $C_{i}$ are $n$-divisible ie there exists divisors
$L_{i}$ such that $C_{i}\sim nL_{i}$. The divisor $v_{i}=\frac{1}{n}C_{i}-L_{i}$
is in $Div^{0}(A,nC)/\sim$, has order $n$ and the multiplicity of
an irreducible component $C'_{i}$ in $\frac{1}{n}C_{i}$ is $\frac{1}{n}$.
The group generated by divisors $\frac{1}{n}C_{i}-L_{i}$ is isomorphic
to $(\ZZ/n\ZZ)^{d}$ and there exists a $(\ZZ/n\ZZ)^{d}$-cover of
$A$ branched with index $n$ over $C$. For the computation of the
Chern numbers of the resolution $X_{n}$ of that cover we refer to
the local analysis of the $(\ZZ/n\ZZ)^{d}$-branched covers of the
plane constructed by Hirzebruch in \cite{Hirz84} (see also the geometric
approach of \cite{Gao}). \\
The following quantities $f_{0},f_{1},f_{2}$ are linear under isogenies:
\[
f_{0}=\sum_{k\geq2}t_{k},\, f_{1}=\sum_{k\geq2}kt_{k},\, f_{2}=\sum_{k\geq2}k^{2}t_{k}.
\]
Let $\pi:Z\to A$ be the blow-up at the $f_{0}-t_{2}=\sum_{k\geq3}t_{k}$
singularities of $C$ of multiplicities $k\geq3$ and let $\bar{C}=\sum\bar{C}_{i}$
be the strict transform of $C$ in $Z$. For a singularity $p$ of
$C$ of multiplicity $k_{p}\geq3$, we denote by $E_{p}\hookrightarrow Z$
the exceptional curve over $p$. There exists a degree $n^{d}$ map
\[
f:X_{n}\to Z
\]
branched with index $n$ above the curve $\bar{C}$. Above $E_{p}$
lies $n^{d-r}$ $(r=k_{p})$ copies in $X_{n}$ of a smooth curve
$F_{p}$ which is a $(\ZZ/n\ZZ)^{r-1}$-cover of $E_{p}$ ramified
with index $n$ at $r$ points, thus 
\[
e(F_{p})=n^{r-1}(2-r)+rn^{r-2}=n^{r-2}(2n+r(1-n)).
\]
Since the Galois group permutes these curves, we have $(F_{p})^{2}=-n^{r-2}$.
If a singularity $p$ of $C$ is a node, then $X_{n}$ is smooth over
$p$ and the fiber of $f$ at $p$ has only $n^{d-2}$ points. 

We have 
\[
e(C)=2-2g+f_{0}-f_{1},\, e(C\setminus Sing(C))=2-2g-f_{1},\, e(A\setminus C)=-e(C)=2g-2+f_{1}-f_{0},
\]
and $C^{2}=\sum C_{i}^{2}+f_{2}-f_{1}=2g-2+f_{2}-f_{1}$. Therefore
we obtain
\[
e(X_{n}\setminus f^{-1}E_{p})=n^{d}e(A\setminus C)+n^{d-1}e(C\setminus Sing(C))+n^{d-2}t_{2}
\]
and

\[
e(X_{n}\setminus f^{-1}E_{p})/n^{d-2}=n^{2}(2g-2+f_{1}-f_{0})+n(2-2g-f_{1})+t_{2}.
\]
Since above each exceptional divisor $E_{p}$ in $Z$, we have $n^{d-k}$
curves with Euler number $e(F_{p})$, we get 
\[
e(X_{n})=e(X_{n}\setminus f^{-1}E_{p})+\sum_{k\geq3}n^{d-2}t_{k}(2n+k(1-n))
\]
thus
\[
e(X_{n})/n^{d-2}=(2g-2+f_{1}-f_{0})n^{2}+2(1-g+f_{0}-f_{1})n+f_{1}-t_{2}.
\]
Let us now compute the canonical divisor : $K_{X_{n}}$ is numerically
equivalent to the pullback of
\[
K=\sum E_{p}+\frac{n-1}{n}(\sum E_{p}+\pi^{*}C-\sum k_{p}E_{p})=\sum_{p}\frac{2n-1+k_{p}(1-n)}{n}E_{p}+\frac{n-1}{n}\pi^{*}C,
\]
thus 
\[
K^{2}=\sum_{k\geq3}-\frac{(2n-1+k(1-n))^{2}}{n^{2}}t_{k}+(\frac{n-1}{n})^{2}C^{2}.
\]
We obtain
\[
K_{X_{n}}^{2}/n^{d-2}=(2g-2+3f_{1}-4f_{0})n^{2}+4(f_{0}-f_{1}-g+1)n-f_{0}+f_{1}+t_{2}+2g-2.
\]
Since it covers an Abelian surface the Kodaira dimension of $X_{n}$
is non negative.  Then we get by using the Miyaoka-Yau inequality
\begin{equation}
(3c_{2}-K_{X_{n}}^{2})/n^{d-2}=(f_{0}+4g-4)n^{2}+2(f_{0}-f_{1}-g+1)n+2f_{1}+f_{0}-4t_{2}-2g+2\geq0.\label{eq:Mi Yau, v1}
\end{equation}
As in \cite{Hirz86}, we will use refinements of the Miyaoka-Yau inequality
for surfaces that contain smooth rational curves and elliptic curves,
ie for $n=2$ or $3$. Let $Y$ be a surface of non negative Kodaira
dimension. Suppose that there exists on $Y$ smooth disjoint elliptic
curves $D_{j}$ and $m$ disjoint $(-2)$-curves, disjoint also from
the curves $D_{j}$, then:
\begin{thm}
\label{thm:(Hirzebruch-,-Miyaoka}(Miyaoka \cite[Cor. 1.3]{Miyaoka}).
We have 
\[
3c_{2}(Y)-K_{Y}^{2}\geq\frac{9}{2}m-\sum(D_{j})^{2}.
\]

\end{thm}
For $n=3,$ we get from equation \ref{eq:Mi Yau, v1}: 
\[
(3c_{2}-K_{X_{3}}^{2})/3^{d-2}=4(7g-7+4f_{0}-f_{1}-t_{2})\geq0
\]
Taking into account the fact that over the $3$-points the surface
contains $3^{d-3}t_{3}$ elliptic curves of self-intersection $-3$,
we can refine that inequality and obtain 
\[
-\frac{f_{1}}{f_{0}}\geq\frac{t_{2}+\frac{1}{4}t_{3}-7g+7}{f_{0}}-4.
\]
For $g=1$, ie $C$ is a configuration of elliptic curves, let $Sing(C)$
be the singularity set of $C$, then: 
\[
H(C,Sing(C))=-\frac{f_{1}}{f_{0}}\geq\frac{t_{2}+\frac{1}{4}t_{3}}{f_{0}}-4.
\]
For $\mathcal{P}$ a (non void) set of points in $A$, one can use
the same demonstration as in \cite[thm 3.3]{BDHH} for the linear
$H$-constant of $\PP^{2}$ to conclude that 
\[
H(C,\mathcal{P})\geq H(C)=-\frac{f_{1}}{f_{0}}\geq\frac{t_{2}+\frac{1}{4}t_{3}}{f_{0}}-4.
\]
Suppose that the bound $-4$ is attained, then $4f_{0}=f_{1}$, $t_{2}=t_{3}=0$
and from equality $\sum_{k\geq5}(k-4)t_{k}=2t_{2}+t_{3}$, we see
that $t_{k}=0$ $\forall k\geq5$. The only possibility is $t_{4}\not=0$,
which indeed exists (see below).

For $n=2,$ the surface $X_{2}$ contains $2^{d-3}t_{3}$ disjoint
$(-2)$-curves and it contains $t_{4}2^{d-4}$ elliptic curves of
self-intersection $-4$, therefore
\[
(3c_{2}-K_{X_{2}}^{2})/2^{d-2}\geq\frac{9}{4}t_{3}+t_{4}
\]
and
\[
10g-10+9f_{0}\geq2f_{1}+4t_{2}+\frac{9}{4}t_{3}+t_{4}.
\]
We obtain the following inequality:
\[
10g-10+t_{2}+\frac{3}{4}t_{3}\geq\sum_{k\geq5}(2k-9)t_{k}.
\]
Using $H(C,Sing(C))=\frac{C^{2}-f_{2}}{f_{0}}=\frac{\sum C_{i}^{2}-f_{1}}{f_{0}}$
and $2g-2=\sum C_{i}^{2}$, we obtain 
\[
H(C,Sing(C))\geq\frac{2t_{2}+\frac{9}{8}t_{3}+\frac{1}{2}t_{4}+8-8g}{f_{0}}-\frac{9}{2}.
\]

\section{The Elliptic $H$-constants for Abelian surfaces}

Let $X$ be a smooth projective surface and $C$ be a configuration
of smooth disjoint elliptic curves on $X$. Let us recall the following
result:
\begin{thm}
(Kobayashi) There exists a torsion free lattice $\Gamma\subset PU(2,1)$
such that $(X,C)$ is a smooth compactification of the ball quotient
surface $\mathbb{B}_{2}/\Gamma$ if and only if $(K_{X}+C)^{2}=3e(X\setminus C)$. 
\end{thm}
Let $C$ be an elliptic curve configuration on a an Abelian surface
$A$. Let $X\to A$ be the blow-up at the singular points of $C$
and let $\bar{C}$ be the strict transform of $C$ in $X$. We obtain:
\begin{cor}
The elliptic curve configuration $C$ on $A$ has $H$-index $H(C)=-4$
if and only if $(X,\bar{C})$ is a smooth compactification of a ball
quotient surface. \\
In that case, the curve $C$ has only $4$-points singularities and
there exists a covering $X_{3}\to A$ branched with order $3$ over
$[3]^{*}C$ such that $X_{3}$ is a smooth compact ball quotient surface:
$c_{1}^{2}(X_{3})=3c_{2}(X_{3})$.\end{cor}
\begin{rem}
Hirzebruch constructed such compact ball quotient surface in \cite{Hirz84}
for the case of the Eisenstein Abelian surface (see below).\end{rem}
\begin{proof}
One computes that $(K_{X}+\bar{C})^{2}=f_{1}-f_{0}$ and $e(X\setminus\bar{C})=f_{0}$.
Therefore, one has equality $(K_{X}+\bar{C})^{2}=3e(X\setminus\bar{C})$
if and only if $4f_{0}=f_{1}$, which is equivalent by Theorem \ref{thm:Refined Miyaoka Sakai inequality}
to $H(C)=-4$ and to the condition that $C$ has only $4$-points.

By equation \ref{eq:Mi Yau, v1}, the value of $(3c_{2}(X_{n})-K_{X_{n}}^{2})/n^{d-2}$
is $f_{0}(n-3)^{2}$, thus for $n=3$, we obtain a compact ball quotient
surface. 
\end{proof}
Let be $j=\frac{-1+i\sqrt{3}}{2}$ with $i^{2}=-1$, we have:
\begin{prop}
The elliptic $H$-constants of $(\CC/\ZZ[i])^{2}$ and $(\CC/\ZZ[j])^{2}$
are equal to $-4$.\end{prop}
\begin{proof}
Hirzebruch \cite{Hirz84} and Holzapfel \cite{Holzapfel} found elliptic
curves arrangements on $(\CC/\ZZ[i])^{2}$ and $(\CC/\ZZ[j])^{2}$
respectively that satisfies Kobayashi's criteria on suitable blow-up.
In the first example, one has $4$ elliptic curves with $t_{4}=1$
and $t_{k}=0,\, k\not=4$, in the second one has $6$ elliptic curves
with $t_{4}=3$ and $t_{k}=0,\, k\not=4$.\end{proof}
\begin{rem}
Hirzebruch and Holzapfel examples are elliptic curves configurations
on Abelian surfaces with only $4$-points singularities. That situation
must be compared with the plane where there do not exist configurations
of $d>6$ lines with $t_{2}=t_{3}=0$ and $t_{d}=t_{d-1}=t_{d-2}$
(by inequality \ref{eq:Hirzebruch} below).
\end{rem}
Let $E,E'$ be two elliptic curves and let $A$ be an Abelian surface. 
\begin{prop}
The $H$-constant and elliptic $H$-constant of $A$ are invariants
of the isogeny class of $A$. Suppose $A$ is isogenous to $E\times E'$.
If $E$ and $E'$ are not isogenous, then $H_{El,A}=-2$. If $E$
and $E'$ are isogenous, then $H_{El,A}\leq-3$.
\end{prop}

\begin{proof}
Let $\phi:A\to B$ be an isogeny between two Abelian surfaces ; it
is an étale map. Let $C$ be a (reduced) curve on $B$, and let $\mathcal{P}$
be a set of points on $B$. By Lemma \ref{lem:Then branch}, $H_{A}(\phi^{*}\mathcal{P},\phi^{*}C)=H_{B}(C,\mathcal{P})$,
thus 
\[
\inf_{C,\mathcal{P}}H_{A}(C,\mathcal{P})\leq\inf_{C,\mathcal{P}}H_{B}(C,\mathcal{P}).
\]
Since there exists an isogeny $\psi:B\to A$ too, we have the reverse
inequality. That holds also for the elliptic $H$-constant, since
the pull-back by an isogeny of a genus one curve is a union of genus
one curves. 

Let $A$ be isogenous to $E\times E'$. Suppose that $E$ and $E'$
are not isogenous. Then a configuration $C$ of elliptic curves on
$E\times E'$ is as follows: 
\[
C=\sum_{k=1}^{m}F_{k}+\sum_{k=1}^{n}F'_{k}
\]
where the $F_{k}$ (resp. $F'_{k}$) are fibers of the fibration of
$E\times E'$ onto $E$ (resp. $E'$). Then one sees that $\bar{C}^{2}/s$
is minimal and equals $-2$ when the blow-up is taken over the $mn$
intersection points of $F_{i}$ and $F'_{j}$.

Suppose that $E$ and $E'$ are isogenous. Since the elliptic $H$-constant
is isogeny invariant, we can suppose that $E=E'$. Let $\Delta$ be
the diagonal in $E\times E$ and let be $F=\{y=0\}$, $F'=\{x=0\}$,
where $x,y$ are the coordinates. Then $C=\Delta+F+F'$ has one $3$-point
in $0$, and in the blow-up at $0$, we get $\bar{C}^{2}=-3$. 
\end{proof}

\section{Remarks on the Harbourne index of curves}

\subsection{Irreducible curves with low $H$-index}

In view of BNC Conjecture \ref{Conjecture BNC}, it would be interesting
to know irreducible curves $C$ with low Harbourne index
\[
H(C)=\min_{\mathcal{P}}H(C,\mathcal{P}),
\]
then by taking the appropriate blow-up one could get very negative
curves and maybe obtain counter examples to that Conjecture. By example,
if one consider the rational curves of degree $d$ in $\mathbb{P}^{2}$
with the maximal number of nodes, the $H$-index $H(C)$ is close
to $-2$ (on that subject, see also Introduction of \cite{BDHH}).
But if one is willing to go down, we must impose singularities with
higher multiplicities on $C$, and it forces $C$ to have ``negative
genus'' ie to be union of (at least) two curves. This explain the
idea to consider union of curves rather than irreducible curves. Moreover
considering reducible curves gives more functorialities to the $H$-constants,
e.g. when one wishes to compare these constants through a dominant
map $f:X\to Y$.

On that problem of constructing infinitely many curves with low $H$-index,
the example of totally geodesic curves on a Shimura surface $X$ is
particularly interesting. If smooth, such a curves $C$ satisfy $C^{2}=-2(g-1)$
(for $g$ the genus) and were considered as possible counter examples
for BNC. But one of the main results of \cite{BHK} is the fact on
$X$, there are at most a finite number of such curves $C$ with $C^{2}<0$.
However, when one take some blow-up of $X$, one do not know the behavior
of $\bar{C}^{2}$ for $\bar{C}$ the strict transform of such a curve
$C$. \\
The singularities of a totally geodesic curve $C$ on a Shimura or
a Ball quotient surface are ordinary, and so are the singularities
of a union of such curves. This follows from the fact that $C$ is
the image of the trace of lines in the universal cover, and singularities
of a union of lines are ordinary. Let $X$ be a Shimura surface and
$C$ be a (irreducible) totally geodesic curve on it. Let $\delta\in\NN$
be defined by $K_{X}C+C^{2}=2g-2+2\delta$. We have $2\delta=f_{2}-f_{1}$
(where $f_{i}=\sum k^{i}t_{k}$, for $t_{k}$ the number of $k$-points).
Since $C$ is a totally geodesic curve, we have $K_{X}C=4g-4>3g-3$
and therefore \cite[Cor. 4.2]{Ciliberto} implies that $C^{2}\sim2\delta$
when the geometric genus $g$ of $C$ goes to $\infty$. Therefore
\[
H(C,Sing(C))\sim-\frac{f_{1}}{f_{0}}
\]
when $g\to\infty$ and since always $f_{1}\geq2f_{0}$, we obtain
a second example of irreducible curves on a surface with 
\[
\liminf_{C}H(C)\leq-2.
\]

\subsection{The $H$-index of line arrangement}

For the case of an arrangement of $d$ lines in $\PP^{2}$ with $t_{d}=t_{d-1}=t_{d-2}=t_{d-3}=0$
and $d\geq6$, Hirzebruch \cite[140]{Hirzebruch} proved the following
inequality 
\begin{equation}
t_{2}+\frac{3}{4}t_{3}\geq d+\sum_{k\geq5}(k-4)t_{k}.\label{eq:ZZBauer}
\end{equation}
Using that results, the authors of \cite{BDHH} obtained the following
inequality for the $H$-index of such a line arrangement $C$: 
\begin{equation}
H(C)\geq B_{1}:=-4+\frac{1}{\sum t_{k}}(2d+t_{2}+\frac{1}{4}t_{3}).\label{eq:bauer}
\end{equation}
A better inequality for the singularities of line arrangements is
given in \cite[eq. (9)]{Hirz86}, which is 
\begin{equation}
t_{2}+\frac{3}{4}t_{3}\geq d+\sum_{k\geq5}(2k-9)t_{k}.\label{eq:Hirzebruch}
\end{equation}

\begin{prop}
Using inequality \ref{eq:Hirzebruch}, we obtain:
\begin{equation}
H(C)\geq B_{2}:=\frac{1}{\sum t_{k}}(\frac{3}{2}d+2t_{2}+\frac{9}{8}t_{3}+\frac{1}{2}t_{4})-\frac{9}{2}.\label{eq:Nouvelle}
\end{equation}
and moreover that inequality \ref{eq:Nouvelle} is sharper than \ref{eq:bauer}
: $B_{2}\geq B_{1}$.
\end{prop}
For the Klein configuration of lines (see \cite[§ 4.1]{BDHH}) and
the Fermat configuration of 18 lines, inequality \ref{eq:Nouvelle}
is even an equality. 
\begin{proof}
Using the combinatorics, we have 
\[
H(C,Sing(C))=\frac{d^{2}-\sum_{k\geq2}k^{2}t_{k}}{\sum t_{k}}=\frac{d-\sum_{k\geq2}kt_{k}}{\sum t_{k}}.
\]
By \ref{eq:Hirzebruch}, we have $t_{2}+\frac{3}{4}t_{3}\geq d+\sum_{k\geq4}(2k-9)t_{k}$,
thus
\begin{equation}
-\sum_{k\geq2}kt_{k}\geq\frac{1}{2}(d+4t_{2}+\frac{9}{4}t_{3}+t_{4}-9\sum t_{k})\label{eq:ZZZ}
\end{equation}
and we obtain inequality \ref{eq:Nouvelle}. The fact that this new
inequality is sharper comes from the fact that inequality \ref{eq:Hirzebruch}
is better than \ref{eq:ZZBauer}.
\end{proof}

\bigskip

\noindent 

{\large{\setlength{\parindent}{0.5in}}}{\large \par}

{\large{Xavier Roulleau,}}{\large \par}

{\large{Laboratoire de Mathématiques et Applications, Université de
Poitiers,}}{\large \par}

{\large{Téléport 2 - BP 30179 - 86962 Futuroscope Chasseneuil, France}}{\large \par}

\texttt{\large{roulleau@math.univ-poitiers.fr}}
\end{document}